\documentclass[reqno]{amsart} 

\usepackage{sagetex}
 
 \usepackage[T1,T2A]{fontenc}
 \usepackage[ansinew]{inputenc}
\usepackage{amssymb} 
\usepackage{upref} 
\usepackage{graphicx}
\usepackage{color}
\usepackage{url}
\usepackage{paralist} 

\usepackage[british]{babel}
\usepackage{bm} 

\usepackage[all]{xy}

\usepackage{hyperref}

\hypersetup{pdfauthor={Matthias Lenz},pdftitle={Interpolation, box splines, and lattice points in zonotopes}, 
  pdfsubject={Zonotopal Algebra}}

\newtheorem{Theorem}{Theorem}

\newtheorem{Proposition}[Theorem]{Proposition}
\newtheorem{Lemma}[Theorem]{Lemma}

\theoremstyle{definition}
\newtheorem{Definition}[Theorem]{Definition}

\newtheorem{Example}[Theorem]{Example}

\theoremstyle{remark}
\newtheorem{Remark}[Theorem]{Remark}

\newcommand{\di}{\mathrm{d}} 
\newcommand{\supp}{\mathop{\mathrm{supp}}\nolimits}

\newcommand{\vol}{\mathop{\mathrm{vol}}}
\newcommand{\abs}[1]{\left|#1\right|}
\newcommand{\rank}{\mathop{\mathrm{rk}}}
\newcommand{\interior}{\mathop{\mathrm{int}}}

\newcommand{\WLOG}{WLOG} 
\newcommand{\st}{s.\,t.\ } 
\newcommand{\ie}{\textit{i.\,e.\ }} 
\newcommand{\eg}{\textit{e.\,g.\ }} 

\newcommand{\N}{\mathbb{N}}
\newcommand{\Z}{\mathbb{Z}}

\newcommand{\R}{\mathbb{R}}

\newcommand{\Hcal}{\mathcal{H}}

\newcommand{\Zcal}{\mathcal{Z}}
\newcommand{\Pcal}{\mathcal{P}}

\newcommand{\spa}{\mathop{\mathrm{span}}}
\newcommand{\cone}{\mathop{\mathrm{cone}}}

\newcommand{\diff}[1]{\frac{\partial}{\partial #1}}

\newcommand{\sym}{\mathop{\mathrm{Sym}}\nolimits} 

\title{Interpolation, box splines, and lattice points in zonotopes
}
\author{Matthias Lenz}
\email{lenz@maths.ox.ac.uk}
\address{%
Mathematical Institute\\
24--29 St Giles'\\
Oxford\\
OX1 3LB\\
United Kingdom
}
\thanks{The author was supported by an ERC starting grant awarded to Olga Holtz and subsequently
 by a Junior Research Fellowship
 of Merton College (University of Oxford).
}
\date{\today}

       

\begin{document}

\begin{abstract}
Let $X$ be a totally unimodular list of vectors in some lattice. Let $B_X$ be the box spline defined by $X$.
Its support is the zonotope $Z(X)$. We show that any real-valued function defined on 
the set of lattice points in the interior of $Z(X)$ can be extended to a function on $Z(X)$ of the form $p(D)B_X$ in a unique way, where $p(D)$ is a differential operator that is contained in the so-called internal $\Pcal$-space. This was conjectured by Olga Holtz and Amos Ron.
We also point out connections between this interpolation problem and matroid theory, including a deletion-contraction decomposition.
 \end{abstract}
\maketitle

\section{Introduction}
Given a set $\Theta=\{u_1,\ldots, u_k\}$ of $k$ distinct points on the real line  
 and a function $f: \Theta \to \R$, it is well-known that there exists a unique polynomial 
 $p_f$ in the space of univariate polynomials of degree at most $k-1$ \st $p_f(u_i)=f(u_i)$ for $i=1,\ldots, k$.

 If $\Theta$ is contained in $\R^d$ for an integer $d\ge 2$, 
the situation becomes more difficult. Not all of the properties of the univariate case can be preserved simultaneously.
 The minimal number $m_\Theta$ \st for every $f:\Theta\to \R$
 there exists  a polynomial $p_f\in\R[x_1,\ldots, x_d]$ of total degree at most $m_\Theta$ that satisfies $p_f(u_i)= f(u_i)$ 
  depends on the geometric configuration of the points in $\Theta$. 
 Furthermore, the interpolating polynomial $p_f$ of degree 
  at most $m_\Theta$ is in general not uniquely determined.
  This is only possible if the dimension of the space 
  of polynomials of degree at most $m_\Theta$ happens to be equal to $k$.
 
 Uniqueness is possible if we choose the interpolating polynomials from a special space.
 Carl de Boor and Amos Ron introduced the \emph{least solution} to the polynomial interpolation problem.
 For an arbitrary finite point set $\Theta\subseteq \R^d$, they construct a space of multivariate polynomials $\Pi(\Theta)$ that
 has dimension $\abs\Theta$ and that contains 
 a unique polynomial interpolating polynomial $p_f$ for every function $f :\Theta\to \R$
 \cite{boor-ron-1990,deBoor-Ron-1992}.

In this paper, we construct a space that contains unique interpolating functions for the special case where $\Theta$ 
 is the set of lattice points in the interior of a zonotope. The space is of a very special nature: 
  it is obtained by applying certain differential operators to the box spline.
  This is interesting because it connects various algebraic and combinatorial 
  structures with interpolation and approximation theory.

 More  information on multivariate polynomial interpolation can be found  in the survey paper \cite{gasca-sauer-2000}.

\medskip

 In this paper, we use the following setup: 
  $U$ denotes a $d$-dimensional real vector space and  $\Lambda\subseteq U$ a lattice.
  Let $X=(x_1,\ldots, x_N) \subseteq \Lambda$ be a finite list of vectors that spans $U$. 
 We assume that $X$ is totally unimodular with respect to $\Lambda$, \ie every basis for $U$ that can be selected from $X$ is also
 a lattice basis.
 The symmetric algebra over $U$ is denoted by $\sym(U)$.
  We fix a basis $s_1,\ldots, s_d$ for the lattice. This makes it possible to identify $\Lambda$ with $\Z^d$,
  $U$ with $\R^d$, $\sym(U)$ with the polynomial ring
   $\R[s_1,\ldots, s_d]$, and $X$ with a $(d\times N)$ matrix. Then
   $X$ is totally unimodular if and only if every non-singular square submatrix of this matrix has determinant $1$ or $-1$.
 A base-free setup is however more convenient when working with quotient vector spaces.

The \emph{zonotope} $Z(X)$ is defined as
 \begin{equation}
 Z(X):= \left\{ \sum_{i=1}^N \lambda_i x_i : 0\le \lambda_i \le 1  \right\}.
 \end{equation}
We denote its set of interior lattice points by $\Zcal_-(X) := \interior(Z(X)) \cap \Lambda$. 
The \emph{box spline} $B_X : U\to \R$ is a piecewise polynomial function that is supported on the zonotope $Z(X)$.
 It is defined by
\begin{equation}
        B_X(u) := \frac{1}{\sqrt{ \det (XX^T)}} \vol\nolimits_{N-d} \left\{ (\lambda_1,\ldots,
         \lambda_N)\in [0,1]^N :  \sum_{i=1}^N \lambda_i x_i = u  \right\}.
\end{equation}
For examples, see Figure~\ref{Figure:SimpleTwoDBoxSpline} and Example~\ref{Example:CardinalBsplines}.
A good reference for box splines and their applications in approximation theory is \cite{BoxSplineBook}.
Our terminology is closer to \cite[Chapter 7]{concini-procesi-book}, where splines are studied from an algebraic 
 point of view.

A vector $u\in U$ defines a linear form $p_x\in \sym(U)$.
For a sublist $Y\subseteq X$, we define $p_Y := \prod_{y\in Y} p_y$. For example, if $Y=((1,0),(1,2))$, then 
 $p_Y=s_1^2 + 2s_1s_2$.
 Now we  define the  
\begin{align}
\text{\emph{central $\Pcal$-space} } \Pcal(X) &:= \spa\{ p_Y :  \rank(X\setminus Y)= \rank(X) \} \\
\text{and the \emph{internal $\Pcal$-space} }   \Pcal_-(X) &:= \bigcap_{x\in X} \Pcal(X\setminus x).
\end{align}
The space $\Pcal_-(X)$ was introduced in \cite{holtz-ron-2011} where it was also shown that the dimension of this space is equal to 
 $\abs{\Zcal_-(X)}$.
The space $\Pcal(X)$ first appeared in approximation theory \cite{akopyan-saakyan-1988,boor-dyn-ron-1991,dyn-ron-1990}. Later, 
 spaces of this type and generalisations were also studied by authors in other fields, \eg
 \cite{ardila-postnikov-2009,berget-2010,holtz-ron-xu-2012,lenz-hzpi-2012,lenz-forward-2012,wagner-1999}.

We will let the elements of $\Pcal_-(X)$ act as differential operators on the box spline.
For $p\in \Pcal_-(X)\subseteq \sym(U)\cong\R[s_1,\ldots,s_r]$, we write $p(D)$ to denote the differential operator obtained from $p$ by replacing 
  the variable $s_i$ by $\diff{s_i}$.

The following proposition ensures that the box spline is sufficiently smooth so that the derivatives
 that appear in the Main Theorem actually exist.
\begin{Proposition}
\label{Proposition:WeakHRwellDefined}
Let $X\subseteq \Lambda\subseteq U\cong \R^d$ be a list of vectors that is totally unimodular
and let $p\in \Pcal_-(X)$. Then
$p(D)B_X$ is a continuous function.
\end{Proposition}
Now we are  ready to state the Main Theorem.
 It was conjectured by  Olga Holtz and Amos Ron \cite[Conjecture 1.8]{holtz-ron-2011}.
\begin{Theorem}[Main Theorem]
\label{Theorem:weakHoltzRon}
Let $X\subseteq \Lambda\subseteq U\cong \R^d$ 
be a list of vectors that is totally unimodular. %
Let $f$ be a real valued function on $\Zcal_-(X)$, 
the set of interior lattice points of the zonotope defined by $X$.

Then there exists a unique polynomial $p\in \Pcal_-(X)\subseteq\R[s_1,\ldots, s_d]$, \st 
 $p(D)B_X$  equals $f$ on $\Zcal_-(X)$.%

\smallskip
Here, $p(D)$ denotes the differential operator obtained from $p$ by replacing  
 the variable $s_i$ by $\diff{s_i}$
and $B_X$  denotes to the box spline defined by $X$.
\end{Theorem}

\begin{Remark}
 Total unimodularity of the list $X$
 is a crucial requirement  in Theorem~\ref{Theorem:weakHoltzRon}. 
 Namely, the dimension of $\Pcal_-(X)$ and $\abs{\Zcal_-(X)}$ agree if and only if $X$ is totally unimodular.
 Note that if one vector in $X$ is multiplied by an integer $\lambda\ge 2$, $\abs{\Zcal_-(X)}$ increases
 while $\Pcal_-(X)$ stays the same.

 Total unimodularity also enables us to make a simple deletion-contraction proof:
 it implies that $\Lambda/x$ is a lattice for all $x\in X$. In general, quotients of lattices may contain torsion elements.
\end{Remark}

\begin{Remark}
We have mentioned above that $\dim(\Pcal_-(X))=\abs{\Zcal_-(X)}$ holds. This 
 is a consequence of a deep connection between the spaces $\Pcal_-(X)$ and 
  $\Pcal(X)$ and matroid theory.
The Hilbert series of these two spaces
 are evaluations of the Tutte polynomial of the matroid defined by $X$ \cite{ardila-postnikov-2009}.
 One can deduce that the Hilbert series of the internal space is equal to the $h$-polynomial of the broken-circuit
complex \cite{brylawski-1977} of the matroid $M^*(X)$ that is  dual to the matroid defined by $X$
and the Hilbert series of the central space equals the
$h$-polynomial of the matroid complex of $M^*(X)$.
 The Ehrhart polynomial of a zonotope that is defined by a totally unimodular matrix
  is also an evaluation of the Tutte polynomial (see \eg \cite{welsh-tutte-1999}).
In summary, for a totally unimodular matrix $X$
\begin{equation}
\label{equation:TutteDimensionPointCount}
 \dim \Pcal_-(X) = \abs{\Zcal_-(X)} = {\mathfrak T}_X(0,1)  \text{ and } 
 \dim \Pcal(X) = \vol(Z(X)) = {\mathfrak T}_X(1,1)
\end{equation}
holds, where ${\mathfrak T}_X$ denotes the Tutte polynomial of the matroid defined by $X$.

It is also interesting to know that the Ehrhart polynomial of an arbitrary zonotope defined by an integer matrix is an evaluation
 of the arithmetic Tutte polynomial \cite{moci-adderio-ehrhart-2012,moci-adderio-2013}.
\end{Remark}

\begin{figure}[t]
\begin{center}
\input{BoxSpline3.pspdftex}
\end{center}
\caption{A very simple two-dimensional example. Here, $X=((1,0),(0,1),(1,1))$, $\Pcal_-(X)=\R$, and $\abs{\Zcal_-(X)}=1$.
}
\label{Figure:SimpleTwoDBoxSpline}
\end{figure}

\subsection*{Organisation of the article}
In Section~\ref{Section:BoxSplineBackground} we will discuss some basic properties of splines.
We will prove the Main Theorem %
in the one-dimensional case in
Section~\ref{Section:CardinalBSplines}. 
In Section~\ref{Section:Smoothness} we will recall the wall-crossing formula for splines and employ it to
prove Proposition~\ref{Proposition:WeakHRwellDefined}.
In Section~\ref{Section:DeletionContraction} we will define deletion and contraction and prove two lemmas
 that will be used in Section~\ref{Section:ExactSequences} in the proof of the Main Theorem.

\section{Splines}
\label{Section:BoxSplineBackground}
In this section we will introduce the multivariate spline and discuss some basic properties of  splines.
 Proofs of the results that we mention here can be found in  \cite[Chapter 7]{concini-procesi-book}
 and some also in \cite{BoxSplineBook}.

\smallskip
 If the convex hull of the vectors in $X$ does not contain $0$, we  define the
 \emph{multivariate spline} (or truncated power) $T_X : U \to \R$ by
\begin{equation} 
 \label{eq:MultSplineVolumeFormula}
 T_X(u) := \frac{1}{\sqrt{\det(XX^T)}}\vol\nolimits_{N- d} %
   \{ (\lambda_1,\ldots, \lambda_N)  \in \R^N_{\ge 0} : \sum_{i=1}^N \lambda_i x_i = u \}. 
\end{equation}
The support of $T_X$ is the \emph{cone} $\cone(X) := \left\{ \sum_{i=1}^N \lambda_i x_i : \lambda_i \ge 0 \right\}$.

Sometimes it is useful to think of the two splines $B_X$ and $T_X$ as distributions.
 In particular,  one can then define 
 the splines for lists $X\subseteq U$ that do not span $U$.
  
 \begin{Remark}
\label{Definition:Splines}
 Let $X\subseteq U\cong \R^r$ be a finite list of vectors.
The multivariate spline $T_X$ and the box spline $B_X$ are distributions that are characterised 
 by the formulae
\begin{align}
 \int_{U}  \varphi(u) B_X(u) \,\di u &= \int_0^1 \cdots \int_0^1 \varphi 
             \left( \sum_{i=1}^N \lambda_ix_i \right) \di \lambda_1 \cdots \di \lambda_N  \\
 \text{and } \int_{U}  \varphi(u) T_X(u) \,\di u &= \int_0^\infty \cdots \int_0^\infty 
	      \varphi \left( \sum_{i=1}^N \lambda_ix_i \right) \di \lambda_1 \cdots \di \lambda_N.
\end{align}
where $\varphi$ denotes a test function.
\end{Remark}
 
\begin{Remark}
Convolutions of splines are again splines. In particular,
\begin{equation}
\label{eq:SplinesAsConvolutions}
  T_X = T_{x_1} * \cdots * T_{x_N} \text{ and } B_X = B_{x_1} * \cdots * B_{x_n}.
\end{equation}
For $x\in X$, differentiation of the two splines in direction $x$ is particularly easy:
\begin{align}
\label{eq:MultSplineDiff}
 D_x T_X &= T_{X\setminus x}  \\ \text{ and }
\label{eq:BoxSplineDiff}
 D_x B_X &= \nabla_x B_{X\setminus x} := B_{X\setminus x} - B_{X\setminus x}(\cdot - x).
\end{align}
\end{Remark}

\begin{Remark}
\label{Remark:SplineBases}
For a basis $C\subseteq U$,
\begin{equation}
\label{eq:multivariateBasis}
  B_C  = \frac{\chi_{Z(C)}}{\abs{\det(C)}}  \text{ and }
  T_C  = \frac{\chi_{\cone(C)}}{\abs{\det(C)}},
\end{equation}
where $\chi_A : U \to \{0,1\}$ denotes the indicator function of the set $A\subseteq U$.
In conjunction with \eqref{eq:SplinesAsConvolutions}, %
  \eqref{eq:multivariateBasis} provides a simple recursive method to calculate the splines.
 \end{Remark}
 
 \begin{Remark}
 The box spline can easily be obtained from the  multivariate spline. 
 Namely,
\begin{equation}
\label{eq:BoxSplineByMultivariateSplines}
B_X(u) =\sum_{S\subseteq X} (-1)^{\abs S} T_X\left(u - a_S\right),
\end{equation}
where $a_S:=\sum_{a\in S} a$.
\end{Remark}

\section{Cardinal $B$-splines}
\label{Section:CardinalBSplines}
In this section we will prove Theorem~\ref{Theorem:weakHoltzRon} in the one-dimensional case.
 This will be the base case for the inductive proof of the Main Theorem in Section~\ref{Section:ExactSequences}.

Let $X_{N}:=(\underbrace{1,\ldots, 1}_{N\text{ times}})\subseteq \Z\subseteq \R^1$.
\WLOG\
every totally unimodular list of vectors in $\R^1$ can be written in this way.

One can easily calculate the corresponding box splines (cf.~Remark~\ref{Remark:SplineBases}):
\begin{equation}
B_{X_{N+1}}(u)= \int_0^1 B_{X_{N}}(u - \tau ) \,\di \tau = \sum_{j=0}^{N+1} \frac{(-1)^{j}}{N!}  \binom{N+1}{j}(u-j)_+^{N},
\end{equation}
where $(u-j)_+^N:=\max(u-j,0)^N$. 
The functions $B_{X_{N+1}}$ are called \emph{cardinal $B$-splines} in the literature (\eg \cite{boor-1976}).

Note that $\Zcal_-(X_{N+1}) = \{1,2,\ldots, N \}$,
\begin{equation*}
\Pcal_{X_{N+1}} = \spa\{ 1, s,\ldots, s^N \},
\text{ and } \Pcal_-(X_{N+1}) = \spa\{1,s,\ldots, s^{N-1} \}. 
\end{equation*}
Hence, in the one-dimensional case, 
 Theorem~\ref{Theorem:weakHoltzRon} is equivalent to the following proposition.
\begin{Proposition}
\label{Proposition:MainThmOneD}
Let $N\in \N$. For every function $f: \{1,\ldots, N\}\to \R$, there exist uniquely determined numbers
 $\lambda_1,\ldots, \lambda_N\in \R$ \st
 \begin{equation}
   \sum_{i=1}^N \lambda_i  D_x^{i-1} B_{X_{N+1}}(j) = f(j) \text{ for } j=1,\ldots, N.
 \end{equation}
\end{Proposition}
Before proving this proposition, we give a  few  simple examples (see also Figure~\ref{Figure:CardinalBSplines}).
\begin{Example}
\label{Example:CardinalBsplines}
\begin{align}
B_{X_{2}}(s) &= s - 2(s-1)_+ + (s-2)_+
\\
B_{X_{3}}(s) &=  \frac 12\left({s^2} -  3(s-1)^2_+ +  3(s-2)^2_+ -  (s-3)^2_+\right) 
\\
B_{X_{4}}(s) &= \frac 16\left({s^3} -  4(s-1)^3_+ + 6(s-2)_+ -  4 (s-3)^3_+ + (s-4)^3_+ \right) 
\end{align}
\begin{sagesilent}
#f=Piecewise([[(0,1),x^2/2],[(1,2),x^2/2-3/2*(x-1)**2],[(2,3),x^2/2-3/2*(x-1)**2+ 3/2*(x-2)**2],[(3,5),x^2/2-3/2*(x-1)**2+ 3/2*(x-2)**2- 1/2*(x-3)**2]])

def posp(x,n) :    # return nth power if positive and zero otherwise
   return max_symbolic(x,0)**n 
   
B2(x) = max_symbolic(x,0) - 2*max_symbolic(x-1,0) + max_symbolic(x-2,0)

B3(x) = posp(x,2)/2 - 3/2*posp(x-1,2) +  3/2*posp(x-2,2) - 1/2*posp(x-3,2)

B4(x) = posp(x,3)/6 - 4/6*posp(x-1,3) + posp(x-2,3) - 4/6*posp(x-3,3) + posp(x-4,3)/6 

#B_{X_{3}}(s) &=  \frac {s^2}2 - \frac 32(s-1)^2_+  + \frac 32(s-2)^2_+ - 
# #  \frac 12 (s-3)^2_+ 
#\\#
#B_{X_{4}}(s) &= \frac {s^3}6 - \frac 46(s-1)^3_+ + \frac 66(s-2)_+ - \frac 46 (s-3)_+ + \frac 16(s-4)_+ 
#\\
#B_{X_{5}}(s) &= \frac {s^4}{24} - \frac 5{24}(s-1)^4_+ + \frac{10}{24}(s-2)^4_+ - \frac {10}{24}(s-3)^4_+ + \frac{10}{24}(s-4)^4_+

p2 = plot(B2,0,2.1,thickness=5,fontsize=24, aspect_ratio=2.1, ymax=1)
p2.axes_width(4)

p3 = plot(B3,0,3.1,thickness=5,fontsize=24, aspect_ratio=3.1, ymax=1)
p3.axes_width(4)

p4 = plot(B4,0,4.1,thickness=5,fontsize=24, aspect_ratio=4.1, ymax=1)
p4.axes_width(4)
\end{sagesilent}
\begin{figure}[t]
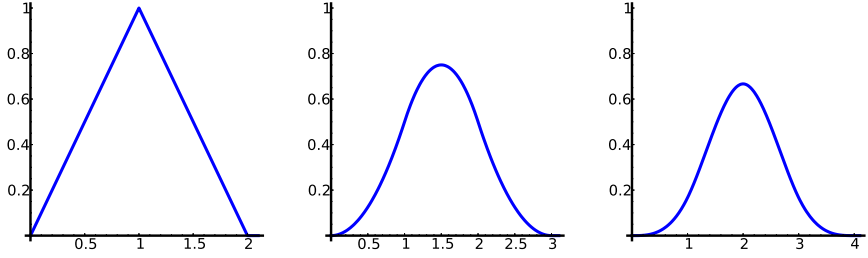
 
\begin{center}
\sageplot[scale=0.23]{p2.plot()}
\sageplot[scale=0.23]{p3.plot()}
\sageplot[scale=0.23]{p4.plot()}
\caption{The cardinal $B$-splines $B_2$, $B_3$, and $B_4$.}
\label{Figure:CardinalBSplines}
\end{center}
\end{figure}
The matrices $M^N$ are defined in \eqref{eq:DefinitionMmatrix} below.
\begin{align*}
 M^2&= \begin{pmatrix}
   1   
\end{pmatrix}
\qquad
\qquad
 M^3= \begin{pmatrix}
   \frac 12 & \frac 12 \\[1mm]
   1 & -1 
\end{pmatrix}
\qquad
\qquad
 M^4= \begin{pmatrix}
   \frac 16 & \frac 46 & \frac 16 \\[1mm]
   \frac 12 & 0 & -\frac 12  \\
   1 & -2 & 1 
\end{pmatrix}
\end{align*}
\end{Example}
\begin{proof}[Proof of Proposition~\ref{Proposition:MainThmOneD}]
For $N\in \N$, we consider the matrix $(N\times N)$-matrix $M^N$ whose entries are given by
\begin{equation}
\label{eq:DefinitionMmatrix}
m_{ij}^N = D_x^{i-1} B_{X_{N+1}}(j).
\end{equation}
The proposition is  equivalent to $M^N$ having full rank.
 The matrix $M^2 = (1)$ obviously has full rank. Let us proceed by induction.  
 By \eqref{eq:BoxSplineDiff}, $D_x^i B_{X_{N+1}} = \nabla_x D_x^{i-1} B_{X_{N}}$. Thus,
 the matrices satisfy the following recursion:
 \begin{equation}
   m_{ij}^N =  m_{i-1,j}^{N-1} - m_{i-1,j-1}^{N-1} \text{ for } i=1,\ldots, N-1,\: j=1,\ldots, N,\text{ and } N\ge 2. 
 \end{equation}
To simplify notation, we set $m_{i0}^{N-1}=m^{N-1}_{i,N}=0$.
Let $v_k,\ldots, v_N$ denote the columns of $M^N$. By induction, they are linearly independent. 
 The columns of $M^{N+1}$ are $(\alpha_1 ,v_1-v_0),\ldots,  (\alpha_N+1, v_{N+1}-v_N)$ with 
 $\alpha_j := B_{X_{N+1}}(j)$. We will now show that these vectors are linearly independent as well.
 Let $\lambda_1\ldots, \lambda_{N+1}\in \R$ \st 
 \begin{equation}
\label{eq:LinearDependence}
 \sum_{j=1}^{N+1}\lambda_j \alpha_j =0
 \end{equation}
 and
 $\lambda_1(v_1-v_0) + \lambda_2(v_2-v_1)+ \ldots + \lambda_N(v_N-v_{N-1}) - \lambda_{N+1} v_N=0$. 
  The latter equation implies that
 all $\lambda_i$ are equal. We conclude that they must all be zero because of \eqref{eq:LinearDependence} and the fact that 
  the $\alpha_j$ are positive.
\end{proof}

\section{Smoothness and wall-crossing}
\label{Section:Smoothness}
The goal of this section is to prove Proposition~\ref{Proposition:WeakHRwellDefined}.
Before doing this, 
we mention some results on the structure of the multivariate spline $T_X$ that are used in the proof. The Wall-Crossing Theorem describes
 the behaviour of $T_X$ when we pass from one region of polynomiality to another.

\begin{Definition}
A \emph{tope} is a connected component of the complement of 
\begin{equation}
\Hcal_X :=\{ \spa(Y) : Y\subseteq X,\, \rank(Y)= \rank(X)-1 \} \subseteq U
\end{equation}
\end{Definition}

The following theorem is  a consequence of  Lemma 3.3 and Proposition 3.7 in \cite{deconcini-procesi-vergne-2010a}.
\begin{Theorem}
\label{Theorem:TXpolynomialontopes}
 Let $X\subseteq U\cong \R^d$ be a  list of vectors $N$ that spans $U$ 
 and whose convex hull does not contain $0$.

 Then $T_X$ agrees with a homogeneous polynomial $f^\tau$ of degree $N-d$
   on every tope $\tau$.%
\end{Theorem}

Given a hyperplane $H$ and a tope $\tau$ which does not intersect $H$ (but its closure may do so),
 we partition $X\setminus H$ into two sets $A_H^{\tau}$ and $B_H^{\tau}$. The set $A_H^{\tau}$ contains the vectors that lie on the same side of
  $H$ as $\tau$ and $B_H^\tau$ contains the vectors that lie on the other side.
 Note that the convex hull of $(A_H^{\tau},-B_H^{\tau})$ does not contain $0$.
 Hence,  we can define the multivariate spline
  \begin{equation}
   T^{\tau}_{X\setminus H} := (-1)^{\abs{B_H^{\tau}}} T_{(A_H^{\tau},-B_H^{\tau})}.
  \end{equation}
Now we are ready to state the wall-crossing formula as in  \cite[Theorem 4.10]{deconcini-procesi-vergne-2010a}.
Related results are in \cite{dahmen-micchelli-1988,szenes-vergne-2003}.
\begin{Theorem}[Wall-crossing for multivariate splines]
\label{Theorem:WallCrossing}
Let $\tau_1$ and $\tau_2$ be two topes whose closures have an $r-1$ dimensional intersection $\tau_{12}$ that spans a 
 hyperplane $H$. Then
  there exists a uniquely determined distribution $ f^{\tau_{12}}$ that is supported on $H$
  \st
    the difference of the local pieces of $T_X$ in $\tau_1$ and $\tau_2$ is equal to the polynomial
 \begin{equation}
   T_X^{\tau_1} - T_X^{\tau_2} = (T_{X\setminus H}^{\tau_1} - T_{X\setminus H}^{-\tau_1})
       * f^{\tau_{12}}.
  \end{equation}
\end{Theorem}
\begin{proof}[Proof of Proposition~\ref{Proposition:WeakHRwellDefined}]
 We will show that $p(D)T_X$ is continuous. By \eqref{eq:BoxSplineByMultivariateSplines}, this implies that
  $p(D)B_X$ is continuous.
 We may always assume that $0$ is not contained in the convex hull of $X$: deleting zeroes from $X$ changes  neither $B_X$ nor $Z(X)$. 
  In addition, one can always multiply a few vectors in $X$ by $-1$ \st all vectors lie on one side of some hyperplane. 
   This is equivalent to a translation of both, $Z(X)$ and $B_X$.

 Let $u\in U$. If $u\in U\setminus \Hcal_X$, there is nothing to prove:  
 by Theorem~\ref{Theorem:TXpolynomialontopes}, $T_X$ is polynomial in a neighbourhood of $u$ and hence smooth.
 If $u\in \Hcal_X$, $u$ is contained in the closure of at least two topes. We have to show that 
 the derivatives of the polynomial pieces in the topes agree on $u$.  
 This can be done using the wall-crossing formula.

 It is sufficient to prove that for two topes $\tau_1$ and $\tau_2$ that have an $(r-1)$-dimensional intersection $\tau_{12}$,
  $p(D)(T_X^{\tau_1} - T_X^{\tau_2})$ vanishes on $\tau_{12}$.

Fix a vector $x\in X\setminus H$. %
 By definition, $\Pcal_-(X)\subseteq \Pcal(X \setminus x)$. %
 This implies that
  $p$ can be written as a linear combination of polynomials $p_Y$ where $Y\subseteq X\setminus x$ and
   $X\setminus (Y \cup x)$ has rank $d$. Hence, $X\setminus (H\cup Y)$ contains at least two vectors.

   By Theorem~\ref{Theorem:WallCrossing},
   \begin{align}
   D_Y (T_{X\setminus H}^{\tau_1} - T_{X\setminus H}^{-\tau_1})
       * f^{\tau_{12}}
  =    (T_{X\setminus (H\cup Y)}^{\tau_1} - T_{X\setminus (H \cup Y)}^{-\tau_1})
       * D_{Y\cap H}f^{\tau_{12}}.
   \end{align}
 This polynomial is the convolution of a distribution supported on $H$
  with the distribution
   $(T_{X\setminus (H\cup Y)}^{\tau_1} - T_{X\setminus (H \cup Y)}^{-\tau_1})$. Since $X\setminus (H\cup Y)$ contains at least
    two elements, this polynomial vanishes on $H$. This finishes our proof.
\end{proof}

\begin{Remark}
Holtz and  Ron conjectured that $\Pcal_-(X)$ is spanned by polynomials $p_Y$ where $Y$ runs 
 over all sublists of $X$ \st 
  $X\setminus (Y\cup x)$ has full rank for all $x\in X$
   \cite[Conjecture 6.1]{holtz-ron-2011}. 
  By formula \eqref{eq:BoxSplineDiff}, this would have implied Proposition~\ref{Proposition:WeakHRwellDefined}. 
 However, this conjecture has recently been disproved \cite{ardila-postnikov-errata-2012}.
\end{Remark}

\section{Deletion and contraction}
\label{Section:DeletionContraction}
In this section we will introduce the operations deletion and contraction which will be used in the proof
 of Theorem~\ref{Theorem:weakHoltzRon} in the next section. We will also prove two lemmas
 about deletion and contraction for box splines and zonotopes.
\smallskip

Let $x\in X$. We call the list $X\setminus x$ the \emph{deletion} of $x$. 
 The image of $X\setminus x$ under the canonical projection $\pi_x : U \to U/\spa(x) =:U/x$ 
 is called the \emph{contraction} of $x$. It is denoted by $X/x$. %

The projection $\pi_x$ induces a map $\sym(\pi_x) : \sym(U)\to \sym(U/x)$.
 If we identify $\sym(U)$ with the polynomial ring $\R[s_1,\ldots,s_r]$ 
  and   $x=s_r$, then
 $\sym(\pi_x)$ is the map from 
 $\R[s_1,\ldots, s_{r}]$ to $\R[s_1,\ldots, s_{r-1}]$ that sends $s_r$ to zero and $s_1,\ldots, s_{r-1}$ to themselves.
 The space $\Pcal(X/x)$ is contained in the symmetric algebra $\sym(U/x)$.

 Since $X$ is totally unimodular,  $\Lambda/x \subseteq U/x$ is a lattice for every $x\in X$ and $X/x$ is  totally unimodular 
  with respect to this lattice. %
\begin{Lemma}
\label{Lemma:DiscreteEqualContinuous}
Let $x\in X$, $u\in U$, and $\bar u=u+\spa(x)$ the coset of $u$ in $X/x$. Then %
\begin{equation} 
\label{eq:BoxSplineContraction}
  B_{X/x}(\bar u) = \int_\R B_{X \setminus x}(u + \tau x) \,\di \tau  %
 =   \sum_{\lambda\in \Z}  B_{X}(u + \lambda x ).
\end{equation}
\end{Lemma}
\begin{proof}
  Let $\bar\varphi: U/x\to \R$ be a test function %
and let $\psi : U\to \R$ be a test function \st $\bar\varphi(\bar u)=\int_\R \psi(u+ \tau x) \,\di \tau$.
Note that a
  distribution $T$ on $U$ that is constant on all cosets of $\spa(x)$ can be identified with a distribution $\bar T$ on $U/x$ 
   via $\bar T(\bar\varphi)=T(\psi)$. This is how \eqref{eq:BoxSplineContraction} can be understood as an equality of distributions.
We may assume that $x=x_N$. Then
\begin{align}
 \int_{U/x} \bar\varphi(\bar u)B_{X/x}(\bar u)\,\di \bar u &=
 \int_0^1 \cdots \int_0^1\bar\varphi( \sum_i \lambda_i \bar x_i) \, \di \lambda_1 \cdots \,\di \lambda_{N-1}  \\
  &=    
 \int_0^1 \cdots \int_0^1\int_\R \psi( \sum_i \lambda_i  x_i  + \tau x) \, \di  \tau \,\di \lambda_1 \cdots \,\di \lambda_{N-1}   \\
  &=    
 \int_U   \psi(u)  %
  {\int_\R B_{X\setminus x}(u +  \tau x)} %
  \,   \di  \tau \,\di u.
\end{align}
This proves the  first equality. 
For the second equality, note that
\begin{equation*}
\int_\R B_{X \setminus x}(u+\tau x) \,\di\tau 
                 = \sum_{\lambda\in\Z}\int_0^1 B_{X \setminus x}(u+ \lambda x - \tau x) \,\di\tau 
                 =  \sum_{\lambda\in\Z}  B_{X}(u+ \lambda x).
                 \qedhere
\end{equation*}
\end{proof}

\begin{Remark}
 Lemma~\ref{Lemma:DiscreteEqualContinuous} is a
  statement on semi-discrete and continuous convolutions with the box spline.
A related result is in \cite{vergne-2011}.
\end{Remark}

The following lemma yields a deletion-contraction formula for the interior points of the zonotope.
See  Figure~\ref{Figure:DelConZonotopes}
 for an illustration.
\begin{Lemma}
\label{Lemma:ZonotopeDelCon}
 The following map is a bijection:
 \begin{align}
  \Zcal_-(X) \setminus \Zcal_-(X\setminus x) &\to \Zcal_-(X/x) \\
          z &\mapsto \bar z.
 \end{align}
\end{Lemma}
\begin{proof}
It is obvious that $\bar z$ is contained in $\Zcal_-(X/x)$.
Using the fact that $\abs{\Zcal_-(X)}$ is an evaluation of the Tutte polynomial
(formula \eqref{equation:TutteDimensionPointCount}) and the deletion-contraction formula for the 
 Tutte polynomial, one can easily establish that  the domain and the range of the map   have the same cardinality.

 Hence it is sufficient to show that the map is injective.
 Let us prove this. First note that $z\in \Zcal_-(X)$ is contained $\Zcal_-(X\setminus x)$ if and only if 
 $z+x\in \Zcal_-(X)$. 
 Let $z_1,z_2\in \Zcal_-(X) \setminus \Zcal_-(X\setminus x)$ \st $\bar z_1=\bar z_2$. This implies that there is 
 a $\lambda\in \R$ \st $z_1=z_2 + \lambda x$.  Because of the total unimodularity, $\lambda$ must be 
  an integer. \WLOG\ $\lambda$ is non-negative.
  By convexity, $z_1,\ldots, z_1+x, \ldots, z_1 +\lambda x $ are contained in $\Zcal_-(X)$.
   By the observation at the beginning of this paragraph, $z_1 + x$ is not 
    contained in $\Zcal_-(X)$. This  implies $\lambda=0$. Hence the map is injective.
\end{proof}

\section{Exact sequences}
\label{Section:ExactSequences}
In this section we will prove the Main Theorem.
We start with a simple observation.
\begin{Remark}
\label{Remark:ColoopsOnly}
If $X$ contains a \emph{coloop}, \ie an element $x$ \st $\rank(X\setminus x)<\rank(X)$, then
 $\Zcal_-(X)=\emptyset$ 
and $\Pcal_-(X) = \{ 0\}$. Hence,
 Theorem~\ref{Theorem:weakHoltzRon} is trivially satisfied. 
 \end{Remark}
 We will consider the set
  $\Xi(X) := \{ f : \Lambda \to \R : \supp(f) \subseteq \Zcal_-(X)  \}$
and the map
\begin{align}
\begin{split}
\gamma_X : \Pcal_-(X)&\to \Xi(X) %
    \\
	      p &\mapsto \bigl[\: \Lambda \ni z \mapsto p(D) B_X(z) \: \bigr]  
\end{split}
	      \end{align}
\begin{figure}[htbp] 
\begin{center}
\input{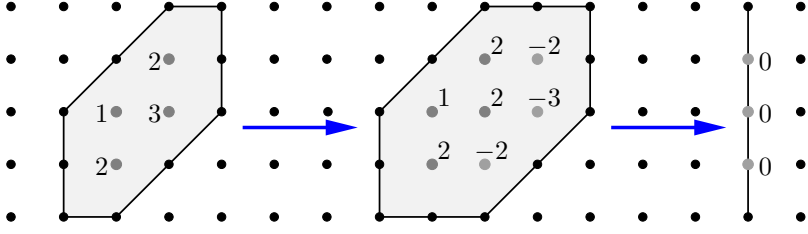}
\end{center}
\caption{Deletion and contraction for a zonotope and a function defined on the interior lattice points of
 the zonotope.
}

\label{Figure:DelConZonotopes}
\end{figure}
\begin{Proposition}
\label{Proposition:ExactSequences}
 Let $d\ge 2$ and let $\Lambda\subseteq U \cong \R^d$ be a lattice. 
 Let $X \subseteq \Lambda$ be a finite list of vectors that spans $U$ 
 and that is totally unimodular with respect to $\Lambda$. 
  Let $x\in X$ be a non-zero element \st $\rank(X\setminus x)=\rank(X)$. 
  
  Then
 the following diagram of real vector spaces is 
 commutative, the rows are exact and the vertical maps are isomorphisms:
\begin{align}
\xymatrix{
  0 \ar[r] &     \Pcal_-(X\setminus x) \ar[r]^{\cdot p_x} \ar[d]^{\gamma_{X\setminus x}}  
   & \Pcal_-(X) \ar[r]^{\sym(\pi_x)} \ar[d]^{\gamma_{X}}
  &  \Pcal_-(X/x )\ar[r] \ar[d]^{\gamma_{X / x}}  &  0
 \\
 0 \ar[r] & \Xi(X\setminus x) \ar[r]^{\nabla_x} &
     \Xi(X) \ar[r]^{\Sigma_x} &  \Xi(X/x) \ar[r]  &  0
  }
\end{align}
\begin{align}
\text{where } \nabla_x(f)(z) &:=  f(z) - f(z - x), \\
   \Sigma_x (f) (\bar z) &:=  \sum_{x\in\bar z \cap \Lambda } f(x)  = \sum_{\lambda\in \Z} f( \lambda x + z) \text{ for some } z\in \bar z.
\end{align}
\end{Proposition}
\begin{proof}%

\emph{Commutativity of the left square:}
Let $z\in \Zcal_-(X)$ and let $p\in \Pcal_-(X\setminus x)$.
By \eqref{eq:BoxSplineDiff},
 $(p\cdot p_x)(D)B_X(z) = \nabla_x (p(D)B_{X\setminus x})(z)$.
Hence $\gamma_X \circ (\cdot p_x) = \nabla_x  \circ \gamma_{X\setminus x}$.

\emph{Commutativity of the right square:}
Let $\bar z\in \Zcal_-(X/x)$, $f\in\Pcal_-(X)$ and let $z\in U$ be a representative of $\bar z$.
 Then
 \begin{equation*}
 \sum_{\lambda \in \Z} p(D) B_X( \lambda x + z)  
  = p(D)\int_\R B_{X\setminus x}(u+ \tau x) \,\di \tau 
  = \sym(\pi_x)(p)(D)B_{X/x}(\bar z)
 \end{equation*}
 because of Lemma~\ref{Lemma:DiscreteEqualContinuous} and the fact that
 applying a differential operator to a function that is constant on a subspace is the same
  as applying the projection of the differential operator  to
   the projection of the function.
Hence $\gamma_{X/x}\circ \sym(\pi_x) = \Sigma_x \circ \gamma_{X}$.

\emph{Exactness:}
Exactness of the first row was  stated in \cite{ardila-postnikov-2009} and proven in \cite{lenz-hzpi-2012}.
 The proof relies on the fact that $\Pcal_-(X)$ can be written as the kernel of a power ideal.
Exactness of the second row is easy to check 
taking into account
Lemma~\ref{Lemma:ZonotopeDelCon}.

\emph{Isomorphisms:}
 By induction over the number of non-zero elements in $X$,  $\gamma_{X\setminus x}$ and $\gamma_{X / x}$ are isomorphisms. 
 Then $\gamma_X$ is also an isomorphism by the five lemma.

Two base cases have to be considered: %
 by deleting elements from $X$, it may happen that $X$ eventually contains only coloops and zeroes. 
  This case is trivial (cf.~Remark~\ref{Remark:ColoopsOnly}).

 By contracting elements from $X$, it may happen that $X$ has rank $1$. We have shown in
  Section~\ref{Section:CardinalBSplines} that
 $\Pcal_-(X)$ and $\Xi(X)$ are isomorphic in this case.
\end{proof}

\begin{proof}[Proof of the Main Theorem]%
The theorem is equivalent 
 to $\gamma_X$ being an isomorphism, which is part of 
Proposition~\ref{Proposition:ExactSequences}.
\end{proof}

\renewcommand{\MR}[1]{} 

\bibliographystyle{amsplain}
\bibliography{../../MasonsConjecture/Mason_Literatur}

\end{document}